 \documentclass[11pt]{amsart}

\usepackage{amssymb}
\usepackage{lineno,hyperref} 

\usepackage{epsfig}  		
\usepackage{epic,eepic}       

 \usepackage{enumerate}
 \usepackage{color}

\newtheorem{thm}{Theorem}[section]

\newtheorem{lem}[thm]{Lemma}
\newtheorem{cor}[thm]{Corollary}

\theoremstyle{definition}
\newtheorem{definition}[thm]{Definition}
\newtheorem{example}[thm]{Example}

\theoremstyle{remark}

\newtheorem{remark}[thm]{Remark}

\numberwithin{equation}{section}

\newcommand{\Q}{\mathbb{Q}}
\newcommand{\C}{\mathbb{C}}

\begin{document}


\title{Ax-Schanuel Type Theorems on Functional Transcendence via Nevanlinna Theory}


\author{Jiaxing Huang}
\address{Department of Mathematics, The University of Hong Kong, 
Pokfulam Road, Hong Kong}
\email{hjxmath@gmail.com}


\author{Tuen-Wai Ng}
\address{Department of Mathematics, The University of Hong Kong, 
Pokfulam Road, Hong Kong}
\email{ntw@maths.hku.hk}



\begin{abstract}
We will apply Nevanlinna Theory to prove several Ax-Schanuel type Theorems for functional transcendence when the exponential map is replaced by other meromorphic functions. We also show that analytic dependence will imply algebraic dependence for certain classes of entire functions. Finally, some links to transcendental number theory and geometric Ax-Schanuel Theorem will be discussed.
\end{abstract}


 \maketitle


\tableofcontents


\section{Introduction and Main Theorems}
The famous Schanuel Conjecture (first appeared in Lang's book \cite{Lang66}) asserts that, 
\emph{given $n$ complex numbers $\alpha_1, \dots, \alpha_n$ which are $\Q$-linearly independent, there are at least $n$ algebraically independent numbers among the $2n$ numbers $\{\alpha_1, \dots, \alpha_n, e^{\alpha_1}, \dots, e^{\alpha_n}\}$.} While this conjecture is still open even for $n=2$, there is a formal  power series analogue proved by Ax using method in differential algebra in 1971 and is now known as the Ax-Schanuel Theorem. 

\begin{thm}[Ax-Schanuel Theorem \cite{Ax71}]\label{thm:Ax}
Let $f_1, \dots, f_n\in\mathbb{C}[[t_1, \dots, t_m]]$ be power series that are $\mathbb{Q}$-linearly independent modulo $\mathbb{C}$. 
Then we have the following inequality: 
$$\mathrm{tr}.\deg_{\mathbb{C}}\mathbb{C}(f_1, \dots, f_n, e(f_1), \dots, e(f_n))\geq n+\mathrm{rank}\left(\frac{\partial f_i}{\partial t_j}\right)_{1\leq j\leq m, 1\leq i\leq n}$$ where $e(x)=e^{2\pi ix}$ and $\mathrm{tr}.\deg_KL$ is the transcendence degree of a field $L$ over its sub-field $K.$
\end{thm}

Notice that we always have $$n+1 \le n+\mathrm{rank}\left(\frac{\partial f_i}{\partial t_j}\right)_{1\leq j\leq m, 1\leq i\leq n} \le 2n$$ and hence  
$n+1 \le \mathrm{tr}.\deg_{\mathbb{C}}\mathbb{C}(f_1, \dots, f_n, e(f_1), \dots, e(f_n)) \le 2n$.

In this paper, we will consider what happens if one replaces the exponential map by other meromorphic or entire function $F$ when each $f_i$ is entire in $\mathbb{C}^m$. We will (for the first time) study the algebraic independence among $f_1, \dots, f_n, F(f_1), \dots, F(f_n)$ via Nevanlinna Theory (instead of differential algebra or o-miminality theory) and obtain the following three main theorems on the estimates of the transcendence degree,
$$\mathrm{tr}.\deg_{\mathbb{C}}\mathbb{C}(f_1, \dots, f_n, F(f_1), \dots, F(f_n))$$
under certain growth assumptions on the Nevenalinna characteristic function $T(r, f_i)$ and the proximity function $m(r,f_i/f_1)$ of the $n$ entire functions $f_1, \dots, f_n$. We also provide examples to illustrate the optimality of these theorems.

\begin{thm}\label{thm:rho}Let $f_1, \dots, f_n$ be entire functions in $\mathbb{C}^m$ satisfying $$T(r, f_i)=S(r, f_{i+1}), \quad \mbox{for}\quad 1\leq i\leq n-1.$$ 
Then for any transcendental meromorphic function $F$ in $\mathbb{C}$, we have 
\begin{equation}\label{eqn:n}
\mathrm{tr}.\deg_{\mathbb{C}}\mathbb{C}(f_1, \dots, f_n, F(f_1), \dots, F(f_n))\geq n+1
\end{equation}

Furthermore, if $f_1, \dots, f_n$ are finite order transcendental entire functions in $\mathbb{C}$, then 
\begin{equation}\label{eqn:2n}
\mathrm{tr}.\deg_{\mathbb{C}}\mathbb{C}(f_1, \dots, f_n, F(f_1), \dots, F(f_n))= 2n
\end{equation}
 for any transcendental entire function $F$ in $\mathbb{C}$ with positive order.
\end{thm}
Since $S(r, f)$ has the growth $o(T(r, f))$ as $r\rightarrow\infty$ outside of a possible exceptional set of finite measure, the condition $T(r, f_i)=S(r, f_{i+1})$ means $T(r,f_{i+1})$ grows much faster than $T(r,f_i)$. The next theorem considers a case which implies $f_1,...,f_n$ have comparable $T$-functions.

\begin{thm}\label{thm:L2} Let $ f_1, \dots, f_n$ be non-constant entire functions in $\mathbb{C}^m$ with $m(r, f_i/f_1)=S(r, f_1)$ for each $i=1, 2, \dots, n$. Suppose that  $f_1, \dots, f_n$ are algebraically independent over $\mathbb{C}$. If the deficiency $\delta(0, f_1)>0$, then we have
$$\mathrm{tr}.\deg_{\mathbb{C}}\mathbb{C}(f_1, \dots, f_n, F(f_1), \dots, F(f_n))=2n$$ for any transcendental entire function $F$ in $\mathbb{C}$.
\end{thm}

\begin{remark} The condition $m(r, f_i/f_1)=S(r, f_1)$  holds when $f_i=f_1^{(i)}$ for any non-constant entire $f_1$ in $\mathbb{C}$ (\cite{Hay64, IIpo11, Ru01}) or when $f_i(z)=f_1(z+\eta_i)$ for finite order entire function $f_1$ in $\mathbb{C}$ and non-zero complex number $\eta_i$ for $i\geq 2$ (\cite{CF08, HK06}). In general $m(r, f_1(\eta_iz)/f_1(z))=S(r, f_1)$ is not true, except when $f_1$ is a zero order entire function. However, in such case, $\delta(0, f_1)=0$ because  $\displaystyle\sum_{a\in\hat{\mathbb{C}}}\delta(a, f)\leq 1$ for any zero order meromorphic function $f$.
\end{remark}

For $n=2$ and $m=1$, we can get rid of the growth restrictions on $T(r, f_i)$ or the proximity function $m(r, f_i/f_1)$ and obtain the following

\begin{thm}\label{thm:prime} 
Let $f_1$ and $f_2$ be entire functions in $\mathbb{C}$. Suppose that $f_1$ and $f_2$ satisfy one of the following conditions:
\begin{enumerate}[(1)]
\item $f_1$ and $f_2$ are two polynomials with distinct degrees;
\item $f_1$ is a polynomial and $f_2$ is a transcendental entire function;
\item Both $f_1$ and $f_2$ are transcendental entire functions which are $\mathbb{C}$-linearly independent modulo $\mathbb{C}$ and $f_1$ is prime.  
\end{enumerate}
 Then we have 
$$\mathrm{tr}.\deg_{\mathbb{C}}\mathbb{C}(f_1, f_2, F(f_1), F(f_2))\geq 2+1$$
for any positive order entire function $F$. 
\end{thm}

\begin{definition}
Let $f$ be a meromorphic function in $\mathbb{C}$,  $f$ is called \emph{prime} if every factorization (in the sense of composition) of the form $f(z)=f_1\circ h(z)$, where $f_1$ is meromorphic and $h$ is entire, implies that either $f_1$ is bilinear or $h$ is linear.  
\end{definition}

Notice that examples of prime entire functions are polynomials of degrees, $e^z+z, ze^z, \sin z e^{\cos z}$, etc (see \cite{CCY90} for more examples). Actually, there are plenty of prime functions as Y. Noda \cite{No81} proved that for any transcendental entire function $f$, $f+\alpha z$ is prime for  all $\alpha\in\mathbb{C}$ except for some countable set $E_f$.

 Applying Theorem \ref{thm:prime}(3), we will give, in Section \ref{sec:geo}, a counter-example to the analogue of a geometric version of Ax-Schanuel Theorem when the exponential map is replaced by other transcendental entire functions. The example illustrates that the validity of a geometric Ax-Schanuel Theorem relies not only on the transcendence of the exponential function, but also on the fact that the exponential function is a uniformization map from $\C$ to $\C^*$.

The rest of this paper is organized as follows. In Section \ref{sec:pre}, we give some definitions in algebra and some results in Nevanlinna theory that we need in the proof of our main results. In Section \ref{sec:main},  we not only proved the main results, but also gave some counter examples to illustrate the necessaries of the assumptions of these results. Finally, links to transcendental number theory and geometric interpretation for Ax-Schanuel Theorem will be discussed in Section \ref{sec:num}. In particular, we will give an example to disprove the validity of a general geometric Ax-Schanuel type inequality. 

\section{Preliminaries}\label{sec:pre}
The main goal of this section is to recall some basic algebraic notions and introduce the concepts and some useful results in Nevanlinna Theory.
\subsection{Nevanlinna Theory}
Let $f$ be a meromorphic function on $\C^m$ and we assume that the reader
is familiar with the following symbols of frequent use in Nevanlinna's theory (see M. Ru \cite{Ru01}):
$$\log^+,\ m(r,f), \ m(r, a, f);\ N(r, f),\ N(r, a, f);\ T(r, f),\ T(r, a, f); \delta(a, f).$$

In certain circumstances of applications of Nevanlinna theory,  we often encounter the quantities which are of growth $o(T(r))$ as $r\rightarrow\infty$ outside of a possible exceptional set of finite linear measure, where $T(r)$ is a continuous, increasing non-negative unbounded function of $r\in\mathbb{R}^+$. Such quantities will be denoted by $S(r)$. In particular, if $T(r)=T(r, f)$, we denote $S(r)$ by $S(r, f)$.

First, we will give some lemmata we need in the proof of our theorems.

 \begin{lem}\label{thm:exg1}
 Let $f$ and $a_j,\ 0\leq j\leq p$, be meromorphic functions on $\mathbb{C}^m$ such that $$\sum_{j=0}^pa_jf^j\equiv 0$$ on $\mathbb{C}^m$. Then $$T(r, f)\leq \sum_{j=0}^p T(r, a_j)+O(1).$$
 \end{lem}
 \begin{proof}
 Following the same argument of Theorem A1.1.6 in M. Ru \cite{Ru01} and the definition of $T$ function,  we can easily obtain the result.
 \end{proof}
 
Now, we present a result on the growth of composite functions first proved by Clunie \cite{Clunie70} and then extended by Chang-Li-Yang \cite{CLY95} to several complex variables.
 
\begin{lem}[Clunie's Lemma \cite{CLY95,Clunie70}]\label{lem:Cl}
Let $f$ be a transcendental entire function on $\mathbb{C}^m$ and let $g$ be a transcendental meromorphic function in the complex plane, then $$T(r, f)=o(T(r, g\circ f))\quad as\quad r\rightarrow\infty$$ and if $g$ is  entire, then $$T(r, g)=o(T(r, g\circ f))\quad as\quad r\rightarrow\infty.$$
\end{lem}
 
Based on Nevanlinna theory, we have the following generalization of Borel's Theorem.

 \begin{lem}[\cite{BCL93, HLY03}]\label{thm:EBL}
  Let $g_j, 0\leq j\leq n$ be entire functions on $\mathbb{C}^m$ such that $g_j-g_k$ are not constants for $0\leq j<k\leq n$ and  $$\sum_{j=0}^na_je^{g_j}\equiv 0$$ where $a_j$'s are meromorphic functions on $\mathbb{C}^m$  such that $T(r, a_j)=o(T(r))$ for $j=0, 1, \dots, n$, hold outside a set with finite measures, and where $$T(r)=\displaystyle\min_{0\leq j<k\leq n}\{T(r, e^{g_j-g_k})\}.$$Then $$a_j\equiv 0,\quad j=0, \dots, n.$$
  \end{lem}
  
The lower order $\lambda(f)$ and order $\rho(f)$ of $f$ are defined as follows:
$$\lambda(f):=\liminf_{r\rightarrow\infty}\frac{\log T(r, f)}{\log r}, \quad \rho(f):= \limsup_{r\rightarrow\infty}\frac{\log T(r, f)}{\log r}.$$

The following lemmata on the growth of meromorphic functions will play important roles in proving our main theorems.

\begin{lem}[\cite{GY72}]\label{lem:GY}Suppose that $f$ and $g$ are entire functions such that $$T(\alpha r, g)=o(T(r, f))\quad \mbox{as}\quad r\rightarrow\infty$$ for some constant $\alpha>1$. Then for any non-constant entire function $F$, $$T(r, F(g))=o(T(r, F(f)))\quad \mbox{as}\quad r\rightarrow\infty.$$
\end{lem}

\begin{lem}[\cite{Zheng}]\label{lem:Zheng} Let $f$ be a meromorphic function with finite order. Given two real numbers $C_1$ and $C_2$ greater than 1, then $$T(C_1r, f)\leq C_2T(r, f)$$ holds outside a set $E$ with finite logarithmic measure.
\end{lem}
The proof of Lemma \ref{lem:Zheng} can be found in Zheng \cite{Zheng} (Lemma 1.1.8). 

\begin{lem}[Edrei and Fuchs \cite{EF64}]\label{lem:EF} Let $f$ be a meromorphic function that is not of zero order and $g$ be a transcendental entire function. Then $f(g)$ is of infinite order.
\end{lem}

\subsection{Algebraic Independence} In this part, we will recall some basic definitions in algebra that we are going to use.

\begin{definition} Let $\psi_1, \dots, \psi_n$ be meromorphic functions in $\mathbb{C}^m$, and let $\mathbb{F}$ be a field. We say that $\psi_1, \dots, \psi_n$ are \emph{$\mathbb{F}$-linearly independent modulo $\mathbb{C}$} if for $(i_1, \dots, i_n)\in\mathbb{F}^n$ and $a\in\mathbb{C}$, the equation $$i_1\psi_1+i_2\psi_2+\dots+i_n\psi_n=a$$ can only be satisfied by $i_1=i_2=\dots=i_n=a=0$.  
 \end{definition}
 
  Let $I=(i_0, i_1, \dots, i_n)$ be a multi-index with $|I|=i_0+i_1+\cdots+i_n$. A polynomial in the variables $u_0, u_1, \dots, u_n$ with functional coefficients in a field $\mathcal{S}$ can always be expressed as $$P(z, u_0, u_1, \dots, u_n)=\sum_{I\in\Lambda}a_I(z)u_0^{i_0}u_1^{i_1}\cdots u_n^{i_n},$$ where the coefficients $a_I$ are functions in $\mathcal{S}$ and $\Lambda$ is an index set.

\begin{definition} Let $f_0, f_1, \dots, f_n$ be meromorphic functions in $\mathbb{C}^m$. We say $f_0, f_1, \dots, f_n$ are \emph{algebraically independent over $\mathcal{S}$} or \emph{$\mathcal{S}$-algebraically independent}, if for any nontrivial polynomial $P(z, u_0, u_1, \dots, u_n)$ in $u_0, \dots, u_n$ with coefficients in $\mathcal{S}$, $P(z, f_0, f_1, \dots, f_n)\not\equiv 0.$

In particular, if $P(z, u_0, \dots, u_n)$ is a linear homogeneous polynomial in $u_0, \dots, u_n$ with coefficients in $\mathcal{S}$, then $f_0, \dots, f_n$ are said to be \emph{linearly independent} over $\mathcal{S}$ or \emph{$\mathcal{S}$-linearly independent}.
\end{definition}

\begin{definition}
Let $L$ be a field and $K\subset L$ a sub-field. A \emph{transcendence basis} for $L$ over $K$ is a maximal algebraically independent over $K$ subset. The \emph{transcendence degree} for $L$ over $K$ ($\mathrm{tr.deg}_{K}L$) is equal to the cardinality of the transcendence basis for $L$ and $K$.
\end{definition}









\section{Proof of Main Theorems}\label{sec:main}
In this section, we will study the Ax-Schanuel type inequalities utilizing the Nevanlinna theory when the exponential map is replaced by a transcendental entire function. 

\subsection{Proof of Theorem \ref{thm:rho}}

\begin{proof}[Proof of Theorem \ref{thm:rho}] 
From Lemma \ref{lem:Cl}, we have $T(r, f_i)=S(r, F(f_i))$ for all $i$. Then we are going to prove that $f_1, \dots, f_n, F(f_n)$ are algebraically independent over $\mathbb{C}$. 

Let $P(u_1, \dots, u_n, v_1)$ be a non-zero polynomial in $u_1, \dots, u_n, v_1$ with constant coefficients. We may write $P(u_1, \dots, u_n, v_1)$ as the following:
$$P(u_1, \dots, u_n, v_1)=\sum_{j=0}^lP_j(u_1, \dots, u_n)v_1^j$$ where $P_j(u_1, \dots, u_n)'s$ are polynomials in $u_1, \dots, u_n$ over $\C$.   Suppose that $P(f_1, \dots, f_n, F(f_n))\equiv 0$ and we denote $P_j(f_1, \dots, f_n)$ by $P_j$. 

It is not difficult to check that $T(r, P_j)=S(r, F(f_n))$ for all $j$, as $T(r, f_i)=S(r, f_{i+1})$ for $1\leq i\leq n-1$, and $T(r, f_n)=S(r, F(f_n))$. 
By Lemma \ref{thm:exg1}, one can conclude that $$T(r, F(f_n))\leq \sum_{j=0}^lT(r, P_j)+O(1)=S(r, F(f_n))$$ which is a contradiction. Thus $P_j\equiv 0$ for all $j$.

Repeating the same argument to each $P_j\equiv 0$, one can deduce that all coefficients of $P(u_1, \dots, u_n, v_1)$ are identically equal to zero.  Therefore, the inequality (\ref{eqn:n}) follows.

\

Next, we will prove the equality (\ref{eqn:2n}). Since each $f_i$ is of finite order, and $F$ is an entire function with positive order, by Lemma \ref{lem:EF}, we have each $F(f_j)$ is of infinite order and hence $$T(r, f_i)=S(r, F(f_j))\quad \mathrm{for}\quad 1\leq i, j \leq n.$$ 

On the other hand, from Lemma \ref{lem:Zheng} and $T(r, f_i)=S(r, f_{i+1})$, we have $$\frac{T(\alpha r, f_i)}{T(r, f_{i+1})}\leq \frac{CT(r, f_i)}{T(r, f_{i+1})}$$ holds outside a set $E$ with finite logarithmic measure, for some constants $\alpha>1$ and $C>1$. Therefore, $$T(\alpha r, f_i)=S(r, f_{i+1}).$$ 
By Lemma \ref{lem:GY}, we have $$T(r, F(f_i))=S(r, F(f_{i+1})), 1\leq i\leq n-1.$$
Hence, for $i=0, \dots, n-1$, $$T(r, f_i)=S(r, f_{i+1}),\quad T(r, f_n)=S(r, F(f_1)) $$ and $$T(r, F(f_i))=S(r, F(f_{i+1})).$$ 

Let $P(u_1, \dots, u_n, v_1, \dots, v_n)$ be a nontrivial polynomial in $u_1, \dots, u_n, v_1, \dots, v_n$ over $\mathbb{C}$ such that $$P(f_1, \dots, f_n, F(f_1), \dots, F(f_n))\equiv 0.$$ Then the equality (\ref{eqn:2n}) follows from the same argument used in the proof of inequality (\ref{eqn:n}).  
\end{proof}

Now, we will give some examples to illustrate the optimality of Theorem \ref{thm:rho}.

\begin{example}\label{rmk:rho}
Let $f_1=z, f_2=e^z$ and $F(z)=e^z$. We have $T(r, f_1)=S(r, f_2)$ and it is easy to verify that $\mathrm{tr}.\deg_{\mathbb{C}}\mathbb{C}(z, e^z, e^z, e^{e^z})=3$. This example shows that the inequality (\ref{eqn:n}) is sharp. 
\end{example}

\begin{example} For equality (\ref{eqn:2n}), the conditions that each $f_i$ is \emph{transcendental} and of \emph{finite order} are necessary. The example in Example \ref{rmk:rho}, shows that the \emph{transcendence} of each $f_i$ is necessary. Another such example is that let $f_1=e^z, f_2=e^{e^z}$ and $F(z)=e^z$, then $f_1$ and $f_2$ are transcendental entire functions, and $\rho(f_1)=1$,  $\rho(f_2)=\infty$,  but $$\mathrm{tr}.\deg_{\mathbb{C}}\mathbb{C}(e^z,e^{e^z},  e^{e^z}, e^{e^{e^z}})=3\neq 4.$$
\end{example}

\subsection{Proof of Theorem \ref{thm:L2}}
We start with a theorem which is an important application of Borel's Theorem (Lemma \ref{thm:EBL}) and Clunie's Lemma, and it will link to the famous Lindemann-Weierstrass Theorem in transcendence number theory.
      
  \begin{thm}\label{thm:ex2}
  Let $g$ be non-constant meromorphic function in $\mathbb{C}^m$ and $\psi_1, \dots, \psi_n$ be meromorphic functions in $\mathbb{C}^m$ such that the following relations hold:
  \begin{enumerate}[1)]
  \item\label{en:1} $\psi_1, \dots, \psi_n$ are linearly independent over $\mathbb{Q}$;
  \item\label{en:2} $T(r, \psi_i)=S(r, g)$ and $\psi_ig$ is entire in $\mathbb{C}^m$, for all $i$;
  \item\label{en:3} $\psi_{k_1}g, \dots, \psi_{k_q}g$  are algebraically independent over $\mathbb{C}$, for $\{k_1, \dots, k_q\}\subset\{1, \dots, n\}$.
  \end{enumerate}
  Then $\psi_{k_1}g, \dots, \psi_{k_q}g, e^{\psi_1g}, \dots, e^{\psi_ng}$ are algebraically independent over $\mathbb{C}$.
  \end{thm}
  
\begin{proof} Suppose there exists a polynomial $P(z_1, \dots, z_q, w_1, \dots, w_n)$ in the variables $z_1, \dots, z_q, w_1, \dots, w_n$ with coefficients in $\mathbb{C}$ such that $$P(\psi_{k_1}g, \dots, \psi_{k_q}g, e^{\psi_1g}, \dots, e^{\psi_ng})\equiv0.$$ We may write $P(z_1, \dots, z_q, w_1, \dots, w_n)$ in the following form:
$$P(z_1, \dots, z_q, w_1, \dots, w_n)=\sum_{i_1, \dots, i_n}p_{i_1, \dots, i_n}(z_1, \dots, z_q)w_1^{i_1}\cdots w_n^{i_n}$$ 
where $p_{i_1, \dots, i_n}(z_1, \dots, z_q)$ is a polynomial in $z_1, \dots, z_q$ with coefficients in $\mathbb{C}$, for all $i_1, \dots, i_n$. 
If $P(\psi_{k_1}g, \dots, \psi_{k_q}g, e^{\psi_1g}, \dots, e^{\psi_ng})\equiv0$, then 
\begin{equation}\label{eqn:Pp}
\sum_{i_1, \dots, i_n}p_{i_1, \dots, i_n}(\psi_{k_1}g, \dots, \psi_{k_q}g)e^{(i_1\psi_1+\dots+i_n\psi_n)g}\equiv0.
\end{equation} 

We then write the equation (\ref{eqn:Pp}) as $$P(\psi_{k_1}g, \dots, \psi_{k_q}g, e^{\psi_1g}, \dots, e^{\psi_ng})=\sum_{i_1, \dots, i_n}p_{i_1, \dots, i_n}e^{h_{i_1, \dots, i_n}g}\equiv 0,$$ 
where $$p_{i_1, \dots, i_n}=p_{i_1, \dots, i_n}(\psi_{k_1}g, \dots, \psi_{k_q}g)$$ 
and $$h_{i_1, \dots, i_n}=i_1\psi_1+\dots+i_n\psi_n.$$  

It is not hard  to see that $T(r, p_{i_1, \dots, i_n})=O(T(r, g))$ and $$T(r, g)=T(r, (h_{i_1, \dots, i_n}-h_{j_1, \dots, j_n})g)=S(r, e^{(h_{i_1, \dots, i_n}-h_{j_1, \dots, j_n})g})$$ for all $(i_1, \dots, i_n)\neq (j_1, \dots, j_n)$, by the assumption (\ref{en:1}) and (\ref{en:2}) and Lemma \ref{lem:Cl}. 
Therefore, applying Lemma \ref{thm:EBL}, we have $$p_{i_1, \dots, i_n}(\psi_{k_1}g, \dots, \psi_{k_q}g)\equiv 0$$ for all $(i_1, \dots, i_n)$. 

On the other hand, $\psi_{k_1}g, \dots, \psi_{k_q}g$ are algebraically independent over $\mathbb{C}$, thus all coefficients of $p_{i_1, \dots, i_n}$ are identically equal to zero. Hence the result follows.
\end{proof}

\begin{cor}\label{cor:ex2} Let $f_1, \dots, f_n$ be entire functions in $\mathbb{C}^m$ such that they are algebraically independent over $\mathbb{C}$. If $m(r, f_i/f_1)=S(r, f_1)$, for $i=1, \dots, n$ and $\delta(0, f_1)=1$, then $f_1, \dots, f_n, e^{f_1}, \dots, e^{f_n}$ are algebraically independent over $\mathbb{C}$.
\end{cor}
\begin{proof} It is not hard to show that $T(r, f_i/f_1)=S(r, f_1)$ if $m(r, f_i/f_1)=S(r, f_1)$ and $\delta(0, f_1)=1$. Let $\psi_i=f_i/f_1, i=1, \dots, n$, $g=f_1$. Applying Theorem \ref{thm:ex2}, we are done.
\end{proof}

Actually, Corollary \ref{cor:ex2} is a precursor of Theorem \ref{thm:L2}. To prove Theorem \ref{thm:L2}, we will first establish the following theorem which shows that analytic dependence will imply algebraic dependence for certain class of entire functions. The proof will follow closely the work of F. Gross and C. F. Osgood \cite{GO96} and B. Q. Li \cite{Li96} on the reduction of an analytic ODE to an algebraic ODE.

\begin{thm}\label{thm:L}
Let $f_1, f_2, \dots, f_n$ be non-constant entire functions in $\mathbb{C}^m$ with $$m(r, f_i/f_1)=S(r, f_1)$$ for each $i=1, 2, \dots, n$. Let $$G(w_1, w_2, \dots, w_n)=\sum_{|I|=0}^{\infty}a_{I}(z)w_1^{i_1}w_2^{i_2}\cdots w_n^{i_n}$$ be a nonzero power series in $\mathbb{C}^n$ where the coefficients $a_I(z)$ are entire functions in $\mathbb{C}^m$ satisfying that 
$$m\left(r, \displaystyle\sum_{|I|=0}^{\infty}|a_I|\right)=S(r, f_1).$$ If $\delta(0,f_1)=0$ and $G(f_1, f_2, \dots, f_n)\equiv 0$, then there exists a nonzero polynomial $P$ with coefficients being polynomials of some $a_I$ (and hence are small functions of $f_1$) such that $$P(f_1, f_2, \dots, f_n)\equiv 0.$$ 
\end{thm}

To prove Theorem \ref{thm:L}, we need the following lemmata (Lemma \ref{lem:ldl} and Lemma \ref{lem:LB}). 

\begin{lem}\label{lem:ldl}
Let $f_1, f_2, \dots, f_n$ be entire functions in $\mathbb{C}^m$ with $m(r, f_i/f_1)=S(r, f_1)$ and $F(w_1, w_2, \dots, w_n)=\displaystyle\sum_{|I|=v}^{\infty}a_{I}(z)w^I$ be a nonzero power series, where $v\geq 0$ and the coefficients $a_I(z)$ are entire functions in $\mathbb{C}^m$ with $m\left(r, \displaystyle\sum_{|I|=v}^{\infty}|a_I|\right)=S(r, f_1)$. If $(f_1, f_2,  \dots, f_n)$ is a solution of the following equation
\begin{equation}\label{eqn:FP}
G(f_1, f_2, \dots, f_n)=P(f_1, f_2, \dots, f_n)
\end{equation} 
where $P$ is a polynomial of degree $u\leq v$ with coefficients being small functions of $f_1$. Then for any $N$ with $u\leq N\leq v$, we have $$m\left(r, \frac{G(f_1, f_2, \dots, f_n)}{f_1^N}\right)=S(r, f_1).$$
\end{lem}
\begin{proof} We shall use an argument similar to the proof of Lemma 4.1 in \cite{Li96}.
Write $$P(f_1, f_2, \dots, f_n)=\sum_{|I|=0}^ub_I(z)f_1^{i_1}f_2^{i_2}\cdots f_n^{i_n}$$ where $b_I$ are small functions of $f_1$. Take any point $z\in\mathbb{C}^m$, we consider the following cases. 

\noindent\textbf{Case (1)} $|f_1(z)|\geq 1$. Then by the equality (\ref{eqn:FP}), 
\begin{eqnarray*} \left|\frac{F(f_1, \dots, f_n)}{f_1^N}(z)\right|
&=&\left|\frac{P(f_1, \dots, f_n)}{f_1^N}(z)\right|\\
&\leq& \sum_{|I|=0}^u\left|b_I(z)\frac{f_1^{i_1}f_2^{i_2}\cdots f_n^{i_n}}{f_1^{|I|}}\right|:=G_1(z).
\end{eqnarray*}

\noindent\textbf{Case (2)} $|f_1(z)|<1$. We divide it into two subcases.

\noindent\textbf{Case (2)(a)} There exists a $j, 1\leq j\leq n$, such that $|f_j(z)|\geq 1$. Then 
\begin{eqnarray*}\left|\frac{G(f_1, \dots, f_n)}{f_1^N}(z)\right|
&\leq& \left|\frac{G(f_1, \dots, f_n)}{f_1^N}(z)f_j^N(z)\right|\\
&=& |P(f_1, \dots, f_n)|\left|\frac{f_j}{f_1}(z)\right|^N\\
&\leq &\left(\sum_{|I|=0}^u\left|b_I(z)\frac{f_1^{i_1}f_2^{i_2}\cdots f_n^{i_n}}{f_1^{|I|}}f_1^{|I|}(z)\right|\right)\left|\frac{f_j}{f_1}(z)\right|^N\\
&\leq &\left(\sum_{|I|=0}^u\left|b_I(z)\frac{f_1^{i_1}f_2^{i_2}\cdots f_n^{i_n}}{f_1^{|I|}}\right|\right)\left(\sum_{j=1}^n\left|\frac{f_j}{f_1}(z)\right|^N\right):=G_2(z).
\end{eqnarray*}

\noindent\textbf{Case (2)(b)} For any $j,\ 1\leq j\leq n,\ |f_j(z)|<1$. Then there exists a $l,\ 1\leq l\leq n$, such that $|f_l(z)|=\displaystyle\max_{1\leq j\leq n}\{|f_j(z)|\}$. In view of the fact $N\leq v$, we have 
\begin{eqnarray*}\left|\frac{G(f_1, \dots, f_n)}{f_1^N}(z)\right|
&\leq& \sum_{|I|=v}^{\infty}\frac{1}{|f_1^N(z)|}|a_I(z)f_l^{|I|}(z)|\\
&=&\sum_{|I|=v}^{\infty}\left|\frac{f_l}{f_1}(z)\right|^N|a_I(z)f_l^{|I|-N}(z)|\\
&\leq &\sum_{|I|=v}^{\infty}\left|\frac{f_l}{f_1}(z)\right|^N|a_I(z)|\\
&\leq &\left(\sum_{l=1}^n\left|\frac{f_l}{f_1}(z)\right|^N\right)\left(\sum_{|I|=v}^{\infty}|a_I(z)|\right):=G_3(z).
\end{eqnarray*}

Combining the above estimations, we have $$\left|\frac{G(f_1, \dots, f_n)}{f_1^N}(z)\right|\leq G_1(z)+G_2(z)+G_3(z)$$ for any $z\in\mathbb{C}^m$. By the assumption that $m(r, f_i/f_1)=S(r, f_1)$ and $m\left(r, \displaystyle\sum_{|I|=v}^{\infty}|a_I|\right)=S(r, f_1)$, we deduce that  $$m\left(r, \frac{G(f_1, f_2, \dots, f_n)}{f_1^N}\right)\leq m(r, G_1+G_2+G_3)=S(r, f_1).$$ This completes the proof.
\end{proof}

The following lemma can be proved by some counting arguments on the number of solutions of certain system of linear equations (a technique often used in transcendental number theory).
 
\begin{lem}[\cite{GO96, Li96}]\label{lem:LB} Let $f_1$ be an entire function in $\mathbb{C}^m$ and $G(w_1, w_2, \dots, w_n)=\displaystyle\sum_{|I|=0}^{\infty}a_I(z)w^I$ be a nonzero series in $w_1, w_2, \dots, w_n$ with coefficients $a_I$ being entire functions in $\mathbb{C}^m$ satisfying that $m\left(r, \displaystyle\sum_{|I|=0}^{\infty}|a_I|\right)=S(r, f_1)$. Then for any integer $L > 0$, there exist three positive integers $p , q$ and $v$ with $p<v/L$, and two non-zero polynomials $P$ and $Q$ in $w_1, w_2, \dots, w_n$, where
\begin{equation}\label{eqn:P}
P(w_1, w_2, \dots, w_n)=\sum_{|I|=0}^pp_I(z)w^I
\end{equation}
and 
\begin{equation}\label{eqn:Q}
Q(w_1, w_2, \dots, w_n)=\sum_{|I|=0}^qq_I(z)w^I, 
\end{equation} such that 
\begin{equation}\label{eqn:QFP}
(QG+P)(w_1, w_2, \dots, w_n)=\sum_{|I|=v}^{\infty}b_I(z)w^I,
\end{equation} where $w=(w_1, w_2, \dots, w_n)$, $p_I$ and $q_I$ are polynomials of some $a_I$, and 
\begin{equation}\label{eqn:mb}
m\left(r, \sum_{|I|=v}^{\infty}|b_I|\right)=S(r, f_1).
\end{equation}
\end{lem}
\begin{proof}[Proof of Theorem \ref{thm:L}]
Let $G(w_1, w_2, \dots, w_n)=\displaystyle\sum_{|I|=0}^{\infty}a_I(z)w^I.$  Also, let $L$ be any positive integer. By Lemma \ref{lem:LB}, we can find integers $p, q$ and $v$ with $p< v/L$, and two nonzero polynomials $P$ and $Q$ with the form (\ref{eqn:P}) and (\ref{eqn:Q}) such that $$(QG+P)(w_1, w_2, \dots, w_n)=\sum_{|I|=v}^{\infty}b_I(z)w^I$$ and the coefficients $b_I$ satisfy (\ref{eqn:mb}). If $G(f_1, f_2, \dots, f_n)\equiv 0$, by (\ref{eqn:QFP}), we have  
\begin{equation}\label{eqn:Pf}
P(f_1, f_2, \dots, f_n)=\sum_{|I|=v}^{\infty}b_I(z)f^I.
\end{equation}
If $P:=P(f_1, f_2, \dots, f_n)\not\equiv 0$, we shall prove that $\delta(0, f_1)=0$. 

First of all,  as $f_i$'s are entire functions, we claim that $T(r, P)\leq pT(r, f_1)+S(r, f_1).$

We first express $P$ as the following $$P=\sum_{k=0}^p\frac{P_k}{f_1^k}f_1^k$$  where $P_k$ is a homogeneous polynomial with degree $k$. As $m\left(r, f_k/f_1\right)=S(r, f_1)$, we have $$m\left(r, \frac{P_k}{f_1^k}\right)=S(r, f_1).$$  Using a theorem of A.Z. Mohon'ko (Theorem 2.25 of \cite{IIpo11}), we have
$$T(r, P)=m(r, P)+S(r, f_1)\leq pm(r, f_1)+S(r, f_1)=pT(r, f_1)+S(r, f_1).$$

Applying Lemma \ref{lem:ldl}  to (\ref{eqn:Pf}), one can conclude that 
\begin{eqnarray*} m(r, 1/f_1)
&\leq& (1/v)m(r, 1/f_1^v)\\
&\leq & (1/v)m\left(r, \frac{P}{f_1^v}\right)+(1/v)m\left(r, \frac{1}{P}\right)\\
&\leq& S(r, f_1)+(1/v)T(r, P)\\
&\leq& (p/v)T(r, f_1)+S(r, f_1)\leq \frac{1}{L}T(r, f_1)+S(r, f_1).
\end{eqnarray*}
Since $L$ can be taken arbitrarily large, we have  $\delta(0, f_1)=0$.  

This completes the proof.
\end{proof}

Now, we are in the position of the proof of Theorem \ref{thm:L2}.

\begin{proof}[Proof of Theorem \ref{thm:L2}] 
Consider a nonzero polynomial $$Q(w_1, \dots, w_n, F(w_1), \dots, F(w_n))$$ with $2n$ complex variables in $w_1, \dots, w_n, F(w_1), \dots, F(w_n)$ over $\mathbb{C}$ such that $$Q(f_1, \dots, f_n, F(f_1), \dots, F(f_n))\equiv 0.$$ 
Let $$G(w_1, \dots, w_n)=Q(w_1, \dots, w_n, F(w_1), \dots, F(w_n)),$$ one can verify that $G$ can be expanded into a nonzero power series with constant coefficients as $F$ is a transcendental entire function. Then from Theorem \ref{thm:L} and the assumptions of $f_1, \dots, f_n$,  there exists a nonzero polynomial $P(z_1, \dots, z_n)$ in $z_1, \dots, z_n$ over $\mathbb{C}$ such that $P(f_1, \dots, f_n)\equiv 0$. On the other hand, $f_1, \dots, f_n$ are assumed to be algebraically independent over $\mathbb{C}$. Hence the result follows.
\end{proof}

We will give examples to show that in Theorem \ref{thm:L2}, the conditions $$m(r, f_i/f_1)=S(r, f_1)\quad \mathrm{and}\quad \delta(0, f_1)>0$$ are necessary. Before presenting the examples, we will state some results we need. 

\begin{thm}[\cite{Ng01}]\label{thm:Ng}
Let $n\geq 1$ and $P(x, y)=\sum_{i=0}^na_i(x)y^i$ be a polynomial in $y$ with entire functions $a_i(x)$ as coefficients such that $a_n\not\equiv 0$. Suppose that $f$ and $g$ are transcendental entire functions such that $P(f, g)\equiv 0$ on $\mathbb{C}$. Then, there exists a transcendental entire function $h$ such that $f=f_1\circ h$ and $g=g_1\circ h$, where $f_1$ and $g_1$ are analytic on the image $\Im(h)$ of $h$.
\end{thm}

\begin{lem}\label{lem:QP} 
Let $f$ be a transcendental entire function. Let $P$ and $Q$ be polynomials such that $P-Q$ is non-constant, then $f+P$ and $f+Q$ are algebraically independent over $\mathbb{C}$.
\end{lem}

\begin{proof} 
Suppose $f+P$ and $f+Q$ satisfy a polynomial equation $R(x, y)=0$ over $\mathbb{C}$. 
By Theorem \ref{thm:Ng}, there exists a transcendental entire function $h$ such that $f+Q=f_1\circ h$ and $f+P=g_1\circ h$ where $f_1$ and $g_1$ are analytic in the image $\Im(h)$ of $h$. Hence we have $Q-P=(f_1-g_1)\circ h$. Since $Q-P$ is non-constant,
 without loss of generality, we may assume that the degree of $Q-P$ is $n$. Since $h$ is transcendental, one can choose $n+1$ distinct points $z_1, \dots, z_{n+1}$ such that $$h(z_1)=\cdots=h(z_{n+1})$$ and hence $$(Q-P)(z_1)=\cdots=(Q-P)(z_{n+1})=a$$ for some $a\in\mathbb{C}$, which is impossible as $Q-P$ is of degree $n$.
\end{proof}

\begin{example}\label{emp:1}
Let $f_1=e^z, f_2=e^{e^z}$, one can check that $f_1$ and $f_2$ are algebraically independent over $\mathbb{C}$ and $\delta(0, f_1)=1$, but $m(r, f_2/f_1)=T(r, f_2)\neq S(r, f_1)$. Let $E(z)=e^z$, then $E(f_1)=e^{e^z}=f_2$ and $E(f_2)=e^{e^{e^z}}$. Therefore, $$\mathrm{tr}.\deg_{\mathbb{C}}\mathbb{C}(f_1, f_2, E(f_1), E(f_2))=\mathrm{tr}.\deg_{\mathbb{C}}\mathbb{C}(f_1, f_2, f_2, E(f_2))=3\neq 2\times 2.$$ Hence the condition $m(r, f_i/f_1)=S(r, f_1)$ is needed.
\end{example}

\begin{example}\label{emp:2}
 Let $f_1=e^z+z, f_2=e^z+1$, then $f_1$ and  $f_2$ are algebraically independent over $\mathbb{C}$ by Lemma \ref{lem:QP} and $m(r, f_2/f_1)=S(r, f_1)$ by the Logarithmic Derivative Lemma, but $\delta(0, f_1)=0$. Let $F(z)=e^z$, then $F(f_1)=e^ze^{e^z}$ and $F(f_2)=ee^{e^z}$, hence $(f_2-1)F(f_2)-eF(f_1)=0$, that is, $$\mathrm{tr}.\deg_{\mathbb{C}}\mathbb{C}(f_1, f_2, F(f_1), F(f_2))\neq 4.$$ Therefore, the condition $\delta(0, f_1)>0$ is also needed. Indeed,  this example also shows that $\delta(0, f_1)>0$ cannot be replaced by $\delta(a, f_1)>0$ where $a$ is non-zero constant or small function of $f_1$.
\end{example}

\subsection{Proof of Theorem \ref{thm:prime}}

To prove Theorem \ref{thm:prime}, we need the following lemma of A. Z. Mohon'ko.
 
\begin{lem}[\cite{Mokh84}]\label{lem:Mokh}
Let $R\in\mathbb{C}[z][u, v]$ be an irreducible polynomial. Let $f$ be a meromorphic solution of the equation $R(z, f(q(z)), f(p(z)))=0$, where $p(z)$ is a polynomial in $z$ with degree $d_p$ and $q(z)$ is a polynomial with degree $d_q$. We write $m:= \deg_{f(q(z))} R$ and $n:= \deg_{f(p(z))} R$ for its degree in $f(q(z))$ and $f(p(z))$ respectively.  
Let $\tau=\dfrac{\log(m/n)}{\log (d_p/d_q)}$. 

If $\tau\geq 1$, then $$\lim_{r\rightarrow\infty}\frac{\log T(r, f)}{\tau\log\log r}=1.$$

If $\tau<1$, then $f$ is a rational function.
\end{lem}

\begin{proof} See A. Z. Mokhon'ko \cite{Mokh84}, Theorem 1 and Remark 2.
\end{proof}

\begin{proof}[Proof of Theorem \ref{thm:prime}]

Suppose $\mathrm{tr}.\deg_{\mathbb{C}}\mathbb{C}(f_1, f_2, F(f_1), F(f_2))<3$, then we consider the following three cases.

\noindent\textbf{Case 1.}  $f_1$ and $f_2$ are two polynomials with distinct degrees.

 Note that $f_1, F(f_1), F(f_2)$ are algebraically dependent over $\mathbb{C}$.
Applying Lemma \ref{lem:Mokh} to $f(z)=F(z), q(z)=f_1(z)$ and $p(z)=f_2(z)$, we have either 
$$\lim_{r\rightarrow\infty}\frac{\log T(r, F)}{\log r}=\lim_{r\rightarrow\infty}\frac{\log T(r, f)}{\tau\log\log r}\frac{\tau\log\log r}{\log r}=0,$$ or $F$ is a rational function. However, $F$ is a transcendental entire function with positive order, thus the result follows.  

\noindent\textbf{Case 2.} $f_1$ is a polynomial and $f_2$ is a transcendental entire function.

 It is not hard to see that $T(r, f_1)=o(T(r, f_2))$ and $T(r, f_2)=o(T(r, F(f_2)))$ by Lemma \ref{lem:Cl}. Therefore, $f_1, f_2$ and $F(f_2)$ are algebraically independent over $\mathbb{C}$, which is impossible as  $\mathrm{tr}.\deg_{\mathbb{C}}\mathbb{C}(f_1, f_2, F(f_1), F(f_2))<3.$

\noindent\textbf{Case 3.} Both $f_1$ and $f_2$ are transcendental entire functions which are $\mathbb{C}$-linearly independent modulo $\mathbb{C}$ and $f_1$ is prime. 

Note that $f_1, f_2$ and $F(f_1)$ are algebraically dependent over $\mathbb{C}$, and there exists a nonzero polynomial $P(z_1, z_2, z_3)$ such that 
$$P(f_1, F(f_1), f_2)=\sum_{i=0}^na_i(f_1, F(f_1))f_2^i\equiv 0$$ 
where $a_i(f_1, F(f_1))$ is polynomial in $f_1, F(f_1)$ and $a_n\not\equiv0$. 
By Theorem \ref{thm:Ng}, there exists a transcendental entire function $h$ such that $f_1=g_1\circ h$ and $f_2=g_2\circ h$, where $g_1$ and $g_2$ are analytic on $\Im(h)$. Since $f_1$ is prime, we have $g_1$ is linear. If $g_2$ is also linear, then $f_1$ and $f_2$ are not $\mathbb{C}$-linear independent modulo $\mathbb{C}$ which contradicts to the assumption. 
Hence $g_2$ is either a polynomial of degree $\geq 2$ or a transcendental entire function. 

If $g_2$ is a transcendental entire function. By Lemma \ref{lem:Cl}, it is not hard to see that $T(r, f_1)=o(T(r, f_2))$ and $T(r, f_2)=o(T(r, F(f_2)))$, as $f_1=g_1\circ h$ and $f_2=g_2\circ h$.  Therefore, $f_1, f_2, F(f_2)$ are algebraically independent over $\mathbb{C}$, which is impossible as $\mathrm{tr}.\deg_{\mathbb{C}}\mathbb{C}(f_1, f_2, F(f_1), F(f_2))<3.$
Hence $g_2$ is a polynomial of degree $\geq 2$.

Note that $f_1, F(f_1), F(f_2)$ are also algebraically dependent over $\mathbb{C}$ and hence there exists a nonzero irreducible polynomial $Q(z_1, z_2, z_3)$ such that $Q(f_1, F(f_1), F(f_2))\equiv 0$. This implies that $$Q(g_1, F(g_1), F(g_2))\circ h\equiv 0.$$ Hence $Q(g_1, F(g_1), F(g_2))\equiv 0$. Notice that $T(r, g_1)=S(r, F(g_1))$ as $F$ is a transcendental entire function. 

 Since $F$ is a transcendental entire function, from Lemma \ref{lem:Mokh}, we have $$\lim_{r\rightarrow\infty}\frac{\log T(r, F)}{\log r}=0,$$ which contradicts with the assumption that the order $\rho(F)>0$ and therefore the result follows.  
\end{proof}

\begin{example}Applying Theorem \ref{thm:prime} to Example \ref{emp:2}, we have $$\mathrm{tr}.\deg_{\mathbb{C}} \mathbb{C}(f_1, f_2, F(f_1), F(f_2))\geq 3.$$ Combining with Example \ref{emp:2}, we have $\mathrm{tr}.\deg_{\mathbb{C}} \mathbb{C}(f_1, f_2, F(f_1), F(f_2))= 3.$
\end{example}

Now, we will give some examples to illustrate the optimality of Theorem \ref{thm:prime}.

In Theorem \ref{thm:prime}, the condition that $f_1$ and $f_2$ are polynomials with \emph{distinct degrees} is necessary. For example,

\begin{example} Let $f_1=z^2$ and $f_2=(z+2\pi)^2$ which are $\mathbb{C}$-linear independent modulo $\mathbb{C}$. Let $F(z)=\cos\sqrt{z}$, then $\cos \sqrt{z}\circ f_1=\cos \sqrt{z}\circ f_2$ where 
$\rho(\cos \sqrt{z})=\frac{1}{2} >0$.
\end{example}

 In Theorem \ref{thm:prime}, the condition of \emph{$\mathbb{C}$-linear independence modulo $\mathbb{C}$} of $f_1$ and $f_2$ cannot be replaced by either \emph{$\mathbb{Q}$-linear independence modulo $\mathbb{C}$} or simply \emph{$\mathbb{C}$-linear independence.}

Let $f$ be a transcendental entire function. Then the first example is as follows.

\begin{example}\label{emp:3}
Let $f_1=\sqrt{-1}f$ and $f_2=f$ which are $\mathbb{Q}$-linearly independent modulo $\mathbb{C}$, consider $F(z)=\cos z^2$, then $F(f_1)=F(f_2)$ and hence $$\mathrm{tr}.\deg_{\mathbb{C}}\mathbb{C}(\sqrt{-1}f, f, \cos f^2, \cos f^2)=2.$$
\end{example}

The second one is to illustrate that \emph{$\mathbb{C}$-linear independence modulo $\mathbb{C}$} cannot be replaced by \emph{$\mathbb{C}$-linear independence.}
\begin{example}\label{emp:4}
 let $f_1=f$ and $f_2=f+c$, where $c$ is a nonzero complex number. Then $f_1$ and $f_2$ are $\mathbb{C}$-linear independent but  not $\mathbb{C}$-linear independent modulo $\mathbb{C}$. Let $F$ be a transcendental entire function with period $c$, then 
 $$\mathrm{tr}.\deg_{\mathbb{C}} \mathbb{C}(f_1, f_2, F(f_1), F(f_2)) =\mathrm{tr}.\deg_{\mathbb{C}} \mathbb{C}(f, f+c, F(f), F(f))=2.$$
 \end{example}
 
Finally, we will show that the \emph{primeness} of $f_1$ is also necessary. 

\begin{example}\label{exm:5} Let $f_1=\sin z$ and $f_2=\cos z$, one can check that both are not prime and they are $\mathbb{C}$-linear independent modulo $\mathbb{C}$. Let $F(z)=\cos (2\pi z^2)$. Then $$F(\sin z)=F(\cos z),$$ and hence we have  
\begin{eqnarray*}
\mathrm{tr}.\deg_{\mathbb{C}} \mathbb{C}(f_1, f_2, F(f_1), F(f_2)) =2.
\end{eqnarray*}
\end{example}

\section{Links to Number Theory and Geometry}\label{sec:num}
In this section, some links to transcendental number theory and geometric interpretation for Ax-Schanuel Theorem will be discussed.
\subsection{Lindemann-Weierstrass Theorem via Nevanlinna Theory}
Let $\alpha$ be a complex number, we say that $\alpha$ is \emph{algebraic} if and only if there exists non-zero polynomial $P(X)\in\mathbb{Q}[X]$ such that $P(\alpha)=0$, otherwise, $\alpha$ is called \emph{transcendental}.

\begin{definition} An analytic function $$f(z)=\sum_{n=0}^{\infty}c_n\frac{z^n}{n!}$$ is said to be an \emph{$E$-function} if:
\begin{enumerate}[1)]
\item all of the $c_n$ lie in an algebraic number field $k$ of finite degree;
\item for any $\epsilon>0$ one has $$|\overline{c_n}|=O(n^{\epsilon n})\quad as \quad n\rightarrow\infty,$$ where $|\overline{a}|$ denotes the maximum modulus of the conjugates of $a$.
\item for any $\epsilon>0$ there exists a sequence of natural numbers $q_1, q_2, \dots,$ with $q_n=O(n^{\epsilon n}),$ such that for all $n$ $$q_nc_j\in\mathbb{Z}_k\quad \mbox{for}\quad 0\leq j\leq n.$$
\end{enumerate}
\end{definition} 

Examples of $E$-functions contain all polynomials with algebraic  coefficients, as well as $e^z, \sin z$ and $\cos z$.

In 1956, Shidlovskii gave a theorem which connected the transcendental number theory and complex function theory as follows (see Chapter 4, \S4 of \cite{Shi89a}).
\begin{thm}[Siegel-Shidlovskii \cite{Shi89a}]\label{thm:Shi} Suppose that the $E$-functions $$f_1(z),\dots, f_n(z), \quad n\geq 1,$$ form a solution of the system of $n$ linear differential equations 
\begin{equation}\label{eqn:LD}
y'_k=Q_{k0}(z)+\sum_{i=1}^nQ_{ki}(z)y_i, \quad k=1, 2, \dots, n, 
\end{equation}
where $Q_{ki}(z)\in\mathbb{C}(z)$. If $\alpha$ is an algebraic number not equal to 0 or a pole of any of the $Q_{ki}(z)$, then $$\mathrm{tr}.\deg_{\mathbb{Q}}\mathbb{Q}(f_1(\alpha), \dots, f_n(\alpha))=\mathrm{tr}.\deg_{\mathbb{C}(z)}\mathbb{C}(z, f_1(z), \dots, f_n(z)).$$
\end{thm}

Applying Theorem \ref{thm:ex2} to $\psi_i=\alpha_i, i=1, \dots, n$ which are algebraic numbers and $g=z$, one has $z, e^{\alpha_1z}, \dots, e^{\alpha_nz}$ are algebraically independent over $\mathbb{C}$. Thus by using Theorem \ref{thm:Shi} with $\alpha=1$, we can also obtain the Lindemann-Weierstrass Theorem. 

\begin{thm}[Lindemann-Weierstrass]
 Let $\alpha_1, \alpha_2, \dots, \alpha_n$ be non-zero algebraic numbers and linearly independent over $\mathbb{Q}$. Then $e^{\alpha_1}, e^{\alpha_2}, \dots, e^{\alpha_n}$ are algebraically independent over $\mathbb{Q}$.
\end{thm}

\subsection{Counter Example to a Geometric Ax-Schanuel Theorem}\label{sec:geo}
We introduce a geometric interpretation of the Ax-Schanuel Theorem following \cite{Tsim15}.

Let $e(x)=e^{2\pi ix}$. Define a holomorphic, non-algebraic map $$\pi_e:\mathbb{C}^n\rightarrow (\mathbb{C}^*)^n, \ \pi_e(z_1, \dots, z_n)=(e(z_1), \dots, e(z_n))$$ where $\mathbb{C}^*=\mathbb{C}\setminus\{0\}$. Let $D_n$ be the graph of $\pi_e$ given by $$D_n=\{(x_1, \dots, x_n, y_1, \dots, y_n)\in\mathbb{C}^n\times(\mathbb{C}^*)^n: \pi_e(x_1, \dots, x_n)=(y_1, \dots, y_n)\}.$$

Denote by $\pi_a$ the projections from $\mathbb{C}^n\times(\mathbb{C}^*)^n$ onto $\mathbb{C}^n$, then the Ax-Schanuel Theorem can be rephrased geometrically as follows:

\begin{thm}[Geometric Ax-Schanuel \cite{Tsim15}]\label{thm:GAxS}
Let $U\subset D_n$ be an irreducible complex analytic subspace such that $\pi_a(U)$ does not lie in the translate of a proper $\mathbb{Q}$-linear subspace of $\mathbb{C}^n$. Then $$\dim_{\mathbb{C}}\mathrm{Zcl}(U)\geq n+ \dim_{\mathbb{C}}U$$ where $\mathrm{Zcl}(U)$ means the Zariski closure of $U$ in $\mathbb{C}^n\times(\mathbb{C}^*)^n$.
\end{thm}

When $U$ is taken to be the image of the map $\textbf{f}: B\rightarrow D_n$ given by $$\textbf{f}(t_1, \dots, t_m)=(f_1, \dots, f_n, e(f_1), \dots, e(f_n)),$$ where $f_i$ are convergent power series in some open neighborhood $B\subset\mathbb{C}^m$, it is easy to verify that $U$ is a complex analytic space and $$\dim_{\mathbb{C}}U=\mathrm{rank}\left(\frac{\partial f_i}{\partial t_j}\right)_{1\leq j\leq m, 1\leq i\leq n}$$
as well as $$\dim_{\mathbb{C}}\mathrm{Zcl}(U)=\mathrm{tr}.\deg_{\mathbb{C}}\mathbb{C}(f_1, \dots, f_n, e(f_1), \dots, e(f_n)).$$
Applying Theorem \ref{thm:GAxS} and Seidenberg embedding theorem \cite{Sei58}, we have the classical Ax-Schanuel Theorem. 

The formulation given in Theorem \ref{thm:GAxS} actually is due to the dubbed Ax-Lindemann by Pila \cite{Pil15}. 

\begin{thm}[Ax-Lindemann]\label{thm:AL}
Let $V\subset(\mathbb{C}^*)^n$ be an algebraic subvariety. Then any maximal algebraic subvariety $W\subset \pi_e^{-1}(V)$ is geodesic,  where A subvariety $W$ of $\mathbb{C}^n$ is called \emph{geodesic} or \emph{weakly special}, if it is defined by any number $l\in\mathbb{N}$ of equations of the form
$$\sum_{i=1}^nq_{ij}z_j=c_i, \quad i=1, \dots, l,$$ where $q_{ij}\in\mathbb{Q}$ and $c_i\in\mathbb{C}$.
\end{thm}

It is easy to see that the Ax-Lindemann Theorem could be viewed as a corollary of the geometric Ax-Schanuel Theorem. Indeed, plugging in $U=(W\times V)\cap D_n$ into Theorem \ref{thm:GAxS}, we see that $U$ has dimension at least as high as that of $W$. Then Theorem \ref{thm:GAxS} implies that $\dim_{\mathbb{C}}V\geq n$ and hence $V$ must be all of $(\mathbb{C}^*)^n$.
  
It is natural to ask if in Theorem \ref{thm:GAxS}, the holomorphic map $\pi_e$ can be replaced by the map $\pi_F: \mathbb{C}^n\rightarrow \mathbb{C}^n$ defined by $\pi_F(z_1, \dots, z_n)=(F(z_1), \dots, F(z_n))$ where $F$ is any transcendental entire function. Unfortunately, Example \ref{exm:5} gives a counterexample to this problem. In other words, the following statement in general does not hold. 

\emph{Let $D$ be the graph of $\pi_F$. Let  $U\subset D$ be an irreducible analytic subspace such that $\pi_a(U)$ does not lie in the translate of a proper $\mathbb{C}$-linear subspace of $\mathbb{C}^n$, where $\pi_a$ is the projection from $\mathbb{C}^n\times\mathbb{C}^n$ onto the first $\mathbb{C}^n$. Then $$ \dim_{\mathbb{C}}\mathrm{Zcl}(U)\geq n + \dim_{\mathbb{C}}U.$$}
\ \ \ \ This is because when $U$ is taken to be the image of the map $\textbf{f}: \mathbb{C}\rightarrow D\subset\mathbb{C}^2\times\mathbb{C}^2$ given by $$\textbf{f}(t)=(\sin t, \cos t,  \cos(2\pi\sin^2t), \cos(2\pi\cos^2t)),$$ then $\dim_{\mathbb{C}}U=1$ and $\dim_{\mathbb{C}}\mathrm{Zcl}(U)=2$ by Example \ref{exm:5}. Hence $$\dim_{\mathbb{C}}\mathrm{Zcl}(U)< 2+ \dim_{\mathbb{C}}U.$$

However, there do exist subsequent Ax-Schanuel type and Ax-Lindemann type results similar to Theorem \ref{thm:GAxS} and Theorem \ref{thm:AL} respectively for  the holomorphic, non-algebraic map $\pi:\Omega\rightarrow X$ where $\Omega$ and $X$ have complex algebraic structure. For example, Ax-Schanuel results are known for affine abelian group varieties in \cite{Ax72}, semi-abelian
varieties \cite{Kir09}, the $j$-function \cite{PT16}, more general Shimura varieties \cite{MPT19}, as well as variations of Hodge structures \cite{BT17}. Also, an Ax-Lindemann result for any Shimura variety has been proved in \cite{KUY16}. \\

\section*{Acknowledgement} The first author was supported by a studentship of HKU. The first and second authors were partially supported by the RGC grant 17301115.

  \bibliographystyle{abbrv}
  \bibliography{bibfile}
\end{document}